%% file: main.tex
\documentclass[11pt]{amsart}
\usepackage[utf8]{inputenc}
\usepackage{amssymb,amsmath}
\usepackage{graphicx}
\usepackage{color}
\usepackage{tikz}
\usepackage{pgfplots}
\usepackage{enumitem}
\usepackage{algorithm}
\usepackage{setspace}
\usepackage{algpseudocode}
\usepackage{mathtools}
\usepackage{xfrac}

\usepackage{accents}
\usepackage{multicol} 
\usepackage{multirow}
\usepackage{soul}
\usepackage{subcaption}
\usepackage{booktabs}
\usepackage{array}
\usepackage{comment}
\newcolumntype{L}[1]{>{\raggedright\let\newline\\\arraybackslash\hspace{0pt}}m{#1}}
\newcolumntype{C}[1]{>{\centering\let\newline\\\arraybackslash\hspace{0pt}}m{#1}}
\newcolumntype{R}[1]{>{\raggedleft\let\newline\\\arraybackslash\hspace{0pt}}m{#1}}
\usepackage{siunitx}
\usepackage{bbold}

\usepackage[hidelinks]{hyperref}
\usepackage[nameinlink,noabbrev]{cleveref}
\newtheorem{theorem}{Theorem}

\theoremstyle{definition}

\theoremstyle{lemma}
\newtheorem{lemma}[theorem]{Lemma}

\theoremstyle{remark}
\newtheorem{remark}[theorem]{Remark}

\Crefname{assumption}{Assumption}{Assumptions}
\numberwithin{theorem}{section}
\numberwithin{equation}{section}
\numberwithin{table}{section}
\numberwithin{figure}{section}
\input{def}
\definecolor{corrRed}{RGB}{18,124,175} 

\allowdisplaybreaks
\begin{document}
\title{A Novel Energy-Based Modeling Framework$^*$}
\author[]{R.~Altmann$^\dagger$, P.~Schulze$^\ddagger$}
\address{${}^{\dagger}$ Institute of Analysis and Numerics, Otto von Guericke University Magdeburg, Universit\"atsplatz 2, 39106 Magdeburg, Germany}
\address{${}^{\ddagger}$ Institute of Mathematics, Technische Universität Berlin, Str.~des 17. Juni~136, 10623 Berlin, Germany}
\email{robert.altmann@ovgu.de, pschulze@math.tu-berlin.de}
\thanks{$^*$ This article will be published in {\em Mathematics of Control, Signals, and Systems}.}
\date{\today}
%
%
\begin{abstract}
We introduce an energy-based model, which seems especially suited for constrained systems. The proposed model generalizes classical port-Hamiltonian input--state--output systems and exhibits similar properties such as energy dissipation as well as structure-preserving interconnection and Petrov--Galerkin projection. In terms of time discretization, the midpoint rule and discrete gradient methods are dissipation-preserving. Besides the verification of these properties, we present ten examples from different fields of application illustrating the great flexibility of the proposed framework. 
\end{abstract}
%
%
\maketitle
%
%
%
%
%
\section{Introduction}

Mathematical modeling of time-dependent systems usually leads to ordinary differential equations (ODEs) or partial differential equations (PDEs).
Moreover, in many applications the governing equations also include algebraic equations which arise from constraints such as coupling conditions or constitutive relations. 
In the particular case of PDEs, constraints may also appear naturally due to structural properties (such as being divergence-free), the inclusion of boundary conditions, or limiting situations (such as the vibrating string with a mass density tending to zero)~\cite{MehZ24}. 
In all these cases, the mathematical model is given by a system of (partial) differential--algebraic equations (DAEs). 
Especially when considering physical systems, the mathematical model typically also has some relevant qualitative properties, as for instance stability or conserved quantities.
Since the violation of these properties may lead to unphysical behavior, it is desirable to preserve these properties when discretizing or approximating the mathematical model.
One possible approach is to preserve the properties by maintaining a certain structure of the governing equations. 
Well-known examples include GENERIC~\cite{GrmO97}, gradient~\cite{HirS74}, Hamiltonian~\cite{Arn89}, or port-Hamiltonian (pH) \cite{SchJ14} structures.

PH systems extend classical Hamiltonian systems by taking internal energy dissipation and energy exchange with the environment or other subsystems into account.
As a consequence, these systems ensure that the corresponding Hamiltonian, which often represents an energy, satisfies a dissipation inequality with a specific supply rate.
Therefore, pH systems are dissipative and, under additional assumptions on the Hamiltonian, also stable; see e.g.~\cite{SchJ14}.
Furthermore, power-preserving interconnections of pH systems lead again to a pH system, which makes the pH formalism especially useful for control \cite{MasS92,VuLM16,SchPFWM21} and for network modeling \cite{Sch04,HauMMMMRS20,AltMU21c}.
Besides, the pH formalism is well-suited for multiphysics applications where energy serves as the \emph{lingua franca} between different physical domains; see e.g.~\cite{VosS14,CarMP17,KriS21}.
Moreover, mathematical models for which pH formulations are already available cover a wide range of application areas including chemistry \cite{HoaCJL11,RamMS13,WanMS18}, electromagnetism \cite{PayMH20,GerHRS21}, mechanics \cite{MacM04,AltS17,BruAPM19a,RasCSS21}, and thermodynamics \cite{EbeM04,LohKL21}. 
While originally introduced for finite-dimensional ODE systems, the pH formalism has been extended to infinite-dimensional systems \cite{JacZ12,RasCSS20,MehZ24} and to port-Hamiltonian differential--algebraic equations (pH-DAEs) \cite{SchM20,MehS23,MehU23}.
In addition, numerical schemes have been derived which preserve pH structures under discretization~\cite{Kot19,BruRS22,Mor24,GieKT24} or model order reduction~\cite{PolS12,GugPBS12,ChaBG16}.

Within this paper, we consider dynamical systems which can be written in the form 
\begin{align}
\label{eq:modelForIntro}
	\begin{bmatrix}
		\partial_{z_1} H \\ \dot{z}_2 \\ 0
	\end{bmatrix}
	=
	(\Jmat-\Rmat)
	\begin{bmatrix}
		\dot{z}_1\\
		\partial_{z_2} H\\
		z_3
	\end{bmatrix}
\end{align}
for some given energy function~$H$ and constant matrices~$\Jmat$, $\Rmat$, possibly with additional inputs or inhomogeneities; see the precise definition in \Cref{sec:model}. 
Such a structure includes classical pH systems (with vanishing $z_1, z_3$) as well as gradient systems (with vanishing $z_2, z_3$). 
The proposed framework seems especially suited for constrained systems with poroelasticity as a prominent example. 
For poroelasticity, pH formulations require the inclusion of additional equations~\cite{AltMU21c}, which is not necessary in the proposed setting.   
Moreover, the state $z_3$ provides more freedom in modeling and, hence, enlarges the applicability of the framework as we will prove by a broad list of examples. 

Within \Cref{sec:model}, we prove that systems of the form~\eqref{eq:modelForIntro} are (just as pH systems) automatically energy dissipative. 
We further show that the interconnection of two such systems preserves the structure if the connection is done in a power-preserving manner and discuss how Petrov--Galerkin projections need to be designed in order to preserve the proposed structure. 
In terms of time stepping schemes, we consider the midpoint scheme as well as discrete gradient methods. 
Both approaches are shown to be dissipation-preserving, i.e., also the discrete energy dissipates for vanishing inputs. 
Finally, \Cref{sec:examples} provides ten examples which fit in the given framework. This includes finite-dimensional examples as well as PDE models with a corresponding structure. 

\subsubsection*{Notation} 
Throughout this paper, we use the Matlab notation for vectors, i.e., we write $[x; y] \coloneqq [x^T, y^T]^T$ for two column vectors~$x,y$.
Besides, we use $\Diag(\Amat_1,\Amat_2,\ldots,\Amat_k)$ to denote (block) diagonal matrices with diagonal (block) entries $\Amat_1,\ldots,\Amat_k$.
To indicate that a matrix $\Amat$ is positive (semi-)definite, we write $\Amat > 0$ ($\Amat \ge 0$).
Moreover, we introduce the empty vector~$\bullet$ in order to distinguish it from possible zero entries. 
With this, we have, e.g., $[x; \bullet] = x$. 
For the Euclidean inner product, we write $\langle x, y\rangle = x^T y$.
%
%
\section{Energy-Based Modeling Framework}\label{sec:model}
This section is devoted to the introduction of an energy-based framework, which is especially suited for constrained systems. 
Consider a state $z = [z_1; z_2; z_3] \in\R^{n_1+n_2+n_3}$, $n\coloneqq n_1+n_2+n_3$, and an energy function $H = H(z_1, z_2)$. 
For a skew-symmetric matrix $\Jmat = -\Jmat^T \in\R^{n, n}$ and a symmetric positive semi-definite matrix~$\Rmat = \Rmat^T \in\R^{n, n}$, we consider the model equations 
\begin{subequations}
\label{eq:model:inputoutput}
\begin{align}
	\begin{bmatrix}
		\partial_{z_1} H \\ 
		\dot{z}_2 \\ 
		0
	\end{bmatrix}
	&=
	(\Jmat-\Rmat)
	\begin{bmatrix}
		\dot{z}_1\\
		\partial_{z_2} H\\
		z_3
	\end{bmatrix} 
	+ 
	\begin{bmatrix} 
		\Bmat_1 \\ 
		\Bmat_2 \\ 
		\Bmat_3 
	\end{bmatrix} u,  \\
	y 
	&= 
	\begin{bmatrix} 
		\Bmat_1^T & \Bmat_2^T & \Bmat_3^T
	\end{bmatrix} 
	\begin{bmatrix} 
		\dot{z}_1 \\ 
		\partial_{z_2} H \\ 
		z_3
	\end{bmatrix}.
\end{align}
\end{subequations}
Therein, $u, y\in \R^{m}$ denote the in- and outputs and $\Bmat_i\in\R^{n_i, m}$ are arbitrary matrices. 
We would like to emphasize that the state $z_3$ is not part of the energy. The introduction of such a state, however, enlarges the class of possible applications as we will illustrate in \Cref{sec:examples}.
\begin{remark}\label{rem:statedependentJR}
For simplicity, we consider constant matrices $\Jmat$, $\Rmat$, and $\Bmat_i$ in \eqref{eq:model:inputoutput}.
However, we emphasize that the results presented in the following extend to the case where these matrices depend on time or the state, as long as $\Jmat$ and $\Rmat$ are pointwise skew-symmetric and symmetric positive semi-definite, respectively.
\end{remark}
\begin{remark}
In the special case $z_3=0$, $\Rmat = 0$, and $u = 0$, system~\eqref{eq:model:inputoutput} reduces to a hybrid representation of a Dirac structure as presented in~\cite[Sect.~5.3]{SchJ14}.
Such representations are obtained by exchanging parts of the effort and flow variables and, thereby, differ from the classical input--state--output representation of pH systems where the flows are expressed in terms of the efforts, cf.~\cite[Sect.~4.2]{SchJ14}. 
\end{remark}
%
%
\subsection{Energy dissipation and interconnections}
Within this section, we summarize properties of system~\eqref{eq:model:inputoutput}. 
Similar as for pH systems, we prove energy dissipation with supply rate $\langle y, u\rangle$ and discuss power-preserving interconnections.  
\begin{lemma}[energy dissipation]
\label{lem:dissipative}
The energy satisfies $\ddt H \le \langle y, u\rangle$. In particular, system~\eqref{eq:model:inputoutput} is energy dissipative for $u = 0$. 
\end{lemma}
\begin{proof}
Due to the assumptions on $\Jmat$ and $\Rmat$, a direct calculation shows
\begin{align*}
	\ddt H
	&= \langle \partial_{z_1} H, \dot{z}_1\rangle + \langle \partial_{z_2} H, \dot{z}_2\rangle \\
	&= \Bigg\langle \begin{bmatrix} \dot{z}_1 \\ \partial_{z_2} H \\ z_3 \end{bmatrix}, \begin{bmatrix} \partial_{z_1} H \\ \dot{z}_2 \\ 0 \end{bmatrix} \Bigg\rangle \\
	&= - \Bigg\langle \begin{bmatrix} \dot{z}_1 \\ \partial_{z_2} H \\ z_3 \end{bmatrix}, \Rmat \begin{bmatrix} \dot{z}_1 \\ \partial_{z_2} H \\ z_3 \end{bmatrix} \Bigg\rangle 
	+ 
	\Bigg\langle \begin{bmatrix} \dot{z}_1 \\ \partial_{z_2} H \\ z_3 \end{bmatrix}, \begin{bmatrix} \Bmat_1 \\ \Bmat_2 \\ \Bmat_3 \end{bmatrix} u \Bigg\rangle
	\le \langle y, u\rangle.  	
\end{align*}
Hence, we have $\ddt H \le 0$ in the case $u=0$. 
\end{proof}
We note that the proof of Lemma~\ref{lem:dissipative} does not only show the dissipation inequality but also the corresponding power balance equation.
Similarly as for pH systems (see, e.g., \cite{MehM19}), the structure of~\eqref{eq:model:inputoutput} is preserved under power-conserving as well as dissipative interconnections as the following lemma shows. 
\begin{lemma}[structure-preserving interconnection]
\label{lem:interconnection}	
Consider two systems of the form~\eqref{eq:model:inputoutput}, namely  
\begin{align*}
	\begin{bmatrix} 
		\partial_{1} H^{[i]} \\ 
		\dot{z}^{[i]}_2 \\ 
		0 
	\end{bmatrix}
	&= 
	\big( \Jmat^{[i]} - \Rmat^{[i]} \big) 
	\begin{bmatrix} 
		\dot{z}^{[i]}_1 \\ 
		\partial_{2} H^{[i]} \\ 
		z^{[i]}_3 
	\end{bmatrix}
	+ 
	\begin{bmatrix} 
		\Bmat_1^{[i]} \\ 
		\Bmat_2^{[i]} \\ 
		\Bmat_3^{[i]}
	\end{bmatrix} u^{[i]}, \\
	y^{[i]} 
	&= 
	\begin{bmatrix} 
		\Bmat_1^{[i]\, T} & \Bmat_2^{[i]\, T} & \Bmat_3^{[i]\, T} 
	\end{bmatrix} 
	\begin{bmatrix} 
		\dot{z}^{[i]}_1 \\ 
		\partial_{2} H^{[i]} \\
		z^{[i]}_3 
	\end{bmatrix}
\end{align*}
with respective states $z^{[1]}$, $z^{[2]}$, energy functions $H^{[1]}$, $H^{[2]}$, and the short notation $\partial_{k} H^{[i]} = \partial_{z^{[i]}_k} H^{[i]}$. Then an interconnection of the form  
\[
	\begin{bmatrix} 
		u^{[1]} \\ 
		u^{[2]} 
	\end{bmatrix}
	= 
	\big( \FmatSkew - \FmatSym \big)
	\begin{bmatrix} 
		y^{[1]} \\ 
		y^{[2]} 
	\end{bmatrix}
	+ 
	\begin{bmatrix} 
		\tilde u^{[1]} \\ 
		\tilde u^{[2]} 
	\end{bmatrix}
\]
with $\FmatSkew=-\FmatSkew^T$ and positive semi-definite $\FmatSym=\FmatSym^T$ yields again a system of the form~\eqref{eq:model:inputoutput}. 	
\end{lemma}
\begin{proof}
We define 
\[
	z_1 
	\coloneqq 
	\begin{bmatrix} 
		z_1^{[1]} \\ 
		z_1^{[2]}
	\end{bmatrix}, \quad 
	z_2 
	\coloneqq 
	\begin{bmatrix} 
		z_2^{[1]} \\ 
		z_2^{[2]} 
	\end{bmatrix}, \quad
	z_3
	\coloneqq 
	\begin{bmatrix} 
		z_3^{[1]} \\ 
		z_3^{[2]} 
	\end{bmatrix}, \quad  
	u 
	\coloneqq 
	\begin{bmatrix} 
		u^{[1]} \\ 
		u^{[2]} 
	\end{bmatrix}, \quad
	\tilde u 
	\coloneqq 
	\begin{bmatrix} 
		\tilde u^{[1]} \\ 
		\tilde u^{[2]} 
	\end{bmatrix}, \quad
	y 
	\coloneqq 
	\begin{bmatrix} 
		y^{[1]} \\ 
		y^{[2]} 
	\end{bmatrix}
\]
and the total energy $H(z) \coloneqq H^{[1]}(z^{[1]}) + H^{[2]}(z^{[2]})$. Assuming the block structure 
\[
	\Jmat^{[i]}
	= 
	\begin{bmatrix} 
		\Jmat^{[i]}_{11} & \Jmat^{[i]}_{12} & \Jmat^{[i]}_{13} \\ 
		-\Jmat^{[i]\, T}_{12} & \Jmat^{[i]}_{22} & \Jmat^{[i]}_{23} \\
		-\Jmat^{[i]\, T}_{13} & -\Jmat^{[i]\, T}_{23} & \Jmat^{[i]}_{33} 
	\end{bmatrix}, \qquad
	\Rmat^{[i]}
	= 
	\begin{bmatrix} 
		\Rmat^{[i]}_{11} & \Rmat^{[i]}_{12} & \Rmat^{[i]}_{13} \\ 
		\Rmat^{[i]\, T}_{12} & \Rmat^{[i]}_{22} & \Rmat^{[i]}_{23} \\
		\Rmat^{[i]\, T}_{13} & \Rmat^{[i]\, T}_{23} & \Rmat^{[i]}_{33} 
	\end{bmatrix},
\]
we define the matrices 
\[
	\Jmat
	\coloneqq 
	\begin{bmatrix} 
		\Jmat^{[1]}_{11} & 0 & \Jmat^{[1]}_{12} & 0 & \Jmat^{[1]}_{13} & 0\\
		0 & \Jmat^{[2]}_{11} & 0 & \Jmat^{[2]}_{12} & 0 & \Jmat^{[2]}_{13} \\  
		-\Jmat^{[1]\, T}_{12} & 0 & \Jmat^{[1]}_{22} & 0 & \Jmat^{[1]}_{23} & 0\\
		0 & -\Jmat^{[2]\, T}_{12} & 0 & \Jmat^{[2]}_{22} & 0 & \Jmat^{[2]}_{23} \\
		-\Jmat^{[1]\, T}_{13} & 0 & -\Jmat^{[1]\, T}_{23} & 0 & \Jmat^{[1]}_{33} & 0\\
		0 & -\Jmat^{[2]\, T}_{13} & 0 & -\Jmat^{[2]\, T}_{23} & 0 & \Jmat^{[2]}_{33} 
	\end{bmatrix}, \qquad 
	\Bmat
	\coloneqq
	\begin{bmatrix} 
		\Bmat^{[1]}_{1} & 0 \\
		0 & \Bmat^{[2]}_{1} \\  
		\Bmat^{[1]}_{2} & 0 \\
		0 & \Bmat^{[2]}_{2} \\
		\Bmat^{[1]}_{3} & 0 \\
		0 & \Bmat^{[2]}_{3} 
	\end{bmatrix},
\]
and $\Rmat$ correspondingly.
%
Since $\Jmat^{[i]}$ is skew-symmetric, so are $\Jmat^{[i]}_{11}$, $\Jmat^{[i]}_{22}$ and $\Jmat^{[i]}_{33}$ for $i=1,2$. 
Hence, $\Jmat$ is skew-symmetric. 
Similarly, it follows that $\Rmat$ is symmetric. 
Moreover, by permuting the block rows and columns of $\Rmat$ as $\lbrace 1,2,3,4,5,6\rbrace \to\lbrace 1,3,5,2,4,6\rbrace$, one observes that $\Rmat$ is congruent to $\Diag(\Rmat^{[1]},\Rmat^{[2]})$ and, therefore, positive semi-definite.
%
%
The two output equations can be combined to
\begin{align*}
	y
	= 
	\begin{bmatrix} y^{[1]} \\ y^{[2]} \end{bmatrix}
	&= 
	\begin{bmatrix} 
		\Bmat_1^{[1]\, T} & 0 & \Bmat_2^{[1]\, T} & 0 & \Bmat_3^{[1]\, T} & 0 \\
		0 & \Bmat_1^{[2]\, T} & 0 & \Bmat_2^{[2]\, T} & 0 & \Bmat_3^{[2]\, T}
	\end{bmatrix} 
	\begin{bmatrix}
		\dot{z}^{[1]}_1 \\ 
		\dot{z}^{[2]}_1 \\ 
		\partial_{2} H^{[1]} \\ 
		\partial_{2} H^{[2]} \\
		z^{[1]}_3 \\
		z^{[2]}_3
	\end{bmatrix} 
	\eqqcolon 
	\Bmat^T
	\begin{bmatrix} 
		\dot{z}_1 \\ 
		\partial_{z_2} H \\
		z_3 
	\end{bmatrix}.
\end{align*}
The interconnection equation reads $u = (\FmatSkew - \FmatSym)\, y + \tilde u$. With this, we get  
\begin{align*}
	\begin{bmatrix} 
		\partial_{z_1} H \\ 
		\dot{z}_2 \\
		0
	\end{bmatrix}
	%
	&=
	\big( \Jmat - \Rmat \big) 
	\begin{bmatrix} 
		\dot{z}_1 \\ 
		\partial_{z_2} H \\
		z_3	 
	\end{bmatrix}
	+ \Bmat\, (\FmatSkew - \FmatSym)\, y 
	+ \Bmat \tilde u \\ 
	&= 
	\Big( \big(\Jmat+\Bmat\FmatSkew\Bmat^T\big) - \big(\Rmat+\Bmat\FmatSym\Bmat^T\big) \Big) 
	\begin{bmatrix} 
		\dot{z}_1 \\ 
		\partial_{z_2} H \\
		z_3
	\end{bmatrix}
	+ \Bmat \tilde u,	
\end{align*}
which is again of the form~\eqref{eq:model:inputoutput}, since $\Bmat\FmatSkew\Bmat^T$ is skew-symmetric and $\Bmat\FmatSym\Bmat^T$ symmetric positive semi-definite. 
\end{proof}
%
%
\subsection{Petrov--Galerkin projection}
Next, we discuss specific projections of the dynamics as considered, e.g., in model order reduction. For this, we restrict ourselves to a quadratic Hamiltonian. More precisely, we assume that 
\begin{align}
\label{eq:quadHamiltonian}
	H 
	= H(z_1,z_2) 
	= \tfrac12\, \langle z_1, \Mmat_1 z_1\rangle + \tfrac12\, \langle z_2, \Mmat_2 z_2\rangle 
\end{align}
with symmetric matrices $\Mmat_1\in\R^{n_1, n_1}$ and $\Mmat_2\in\R^{n_2, n_2}$.
\begin{lemma}[structure-preserving Petrov--Galerkin projection]
\label{lem:Galerkin} 
Consider a quadratic Hamiltonian as in~\eqref{eq:quadHamiltonian}. 
Further assume full-rank matrices 
\[
	\Vmat_1\in \R^{n_1, r_1}, \qquad
	\Vmat_2\in \R^{n_2, r_2}, \qquad
	\Vmat_3\in \R^{n_3, r_3}	
\]
with $r_j\le n_j$, $j=1,2,3$, spanning approximation spaces of lower dimension such that $\Vmat_2^T \Mmat_2 \Vmat_2$ is invertible. 
Then the Petrov--Galerkin approximation, which is obtained by considering the orthogonal complement to $\Vmat_1$, $\Mmat_2\Vmat_2$, and~$\Vmat_3$, respectively, is again a solution of a system of the form~\eqref{eq:model:inputoutput}. 
\end{lemma}
\begin{proof}
For the solution $[z_1;z_2;z_3]\in\R^{n_1+n_2+n_3}$, we consider (Galerkin) approximations defined through $[w_1;w_2;w_3]\in\R^{r_1+r_2+r_3}$ via 
\[
	z_1 
	\approx \Vmat_1 w_1, \qquad
	z_2 
	\approx \Vmat_2 w_2, \qquad
	z_3 
	\approx \Vmat_3 w_3.
\]
The corresponding energy reads $\widetilde H = \frac12\, \langle\Vmat_1 w_1, \Mmat_1 \Vmat_1 w_1\rangle + \frac12\, \langle \Vmat_2 w_2, \Mmat_2 \Vmat_2 w_2\rangle$ and the approximate solutions solve the reduced dynamics which are obtained by restricting the test spaces. For this, we require the residual to be orthogonal to $\Vmat_1$, $\Mmat_2\Vmat_2$, and~$\Vmat_3$, respectively. Hence, $[w_1;w_2;w_3]$ solves the reduced system 
\begin{align*}
	&\begin{bmatrix} 
		\Vmat_1^T & 0 & 0 \\ 
		0 & \Vmat_2^T\Mmat_2 & 0 \\ 
		0 & 0 & \Vmat_3^T 
	\end{bmatrix}
	\begin{bmatrix} 
		\Mmat_1 \Vmat_1w_1 \\ 
		\Vmat_2\dot{w}_2 \\ 
		0 
	\end{bmatrix}\\
	&= 
	\begin{bmatrix} 
		\Vmat_1^T & 0 & 0 \\ 
		0 & \Vmat_2^T\Mmat_2 & 0 \\ 
		0 & 0 & \Vmat_3^T 
	\end{bmatrix}
	(\Jmat - \Rmat) 
	\begin{bmatrix} 
		\Vmat_1 & 0 & 0 \\ 
		0 & \Mmat_2 \Vmat_2 & 0 \\ 
		0 & 0 & \Vmat_3 
	\end{bmatrix}
	\begin{bmatrix} 
		\dot{w}_1 \\ 
		w_2 \\ 
		w_3 
	\end{bmatrix} 
	+
	\begin{bmatrix} 
		\Vmat_1^T & 0 & 0 \\ 
		0 & \Vmat_2^T\Mmat_2 & 0 \\ 
		0 & 0 & \Vmat_3^T 
	\end{bmatrix}
	\begin{bmatrix} 
		\Bmat_1 \\ 
		\Bmat_2 \\ 
		\Bmat_3 
	\end{bmatrix} u.
\end{align*}
To recover the form~\eqref{eq:model:inputoutput}, we introduce~$\widetilde w = [w_1; \Vmat_2^T\Mmat_2\Vmat_2 w_2; w_3]$ as the new state. 
This yields $\partial_{\widetilde w_1} \widetilde H = \Vmat_1^T \Mmat_1 \Vmat_1 \widetilde w_1$ and $\partial_{\widetilde w_2} \widetilde H = (\Vmat_2^T \Mmat_2 \Vmat_2)^{-1} \widetilde w_2 = w_2$.
With this, the reduced system has the form 
\[
	\begin{bmatrix} 
		\partial_{\widetilde w_1} \widetilde H \\ 
		\dot{\widetilde w}_2 \\ 
		0 
	\end{bmatrix}
	= 
	( \widetilde \Jmat - \widetilde \Rmat ) 
	\begin{bmatrix} 
		\dot{\widetilde w}_1 \\ 
		\partial_{\widetilde w_2} \widetilde H \\ 
		\widetilde w_3 
	\end{bmatrix}
	+
	\begin{bmatrix} 
		\widetilde\Bmat_1 \\ 
		\widetilde\Bmat_2 \\ 
		\widetilde\Bmat_3 
	\end{bmatrix} u
\]
with~$\widetilde \Jmat = \Diag(\Vmat_1,\Mmat_2 \Vmat_2,\Vmat_3)^T \Jmat\, \Diag(\Vmat_1,\Mmat_2 \Vmat_2,\Vmat_3)$ being skew-symmetric and, similarly, $\widetilde \Rmat = \Diag(\Vmat_1,\Mmat_2 \Vmat_2,\Vmat_3)^T\, \Rmat\, \Diag(\Vmat_1,\Mmat_2 \Vmat_2,\Vmat_3)$ being symmetric positive semi-definite.
Moreover, we have $[\widetilde\Bmat_1; \widetilde\Bmat_2; \widetilde\Bmat_3] = \Diag(\Vmat_1,\Mmat_2 \Vmat_2,\Vmat_3)^T [\Bmat_1; \Bmat_2; \Bmat_3]$ and the new output equation 
\[
	\widetilde y 
	= 
	\begin{bmatrix} 
		\widetilde\Bmat_1^T & \widetilde\Bmat_2^T & \widetilde\Bmat_3^T
	\end{bmatrix} 
	\begin{bmatrix} 
		\dot{\widetilde w}_1 \\ 
		\partial_{\widetilde w_2} \widetilde H \\ 
		\widetilde w_3 
	\end{bmatrix}.
	\qedhere
\]
\end{proof}
\begin{remark}
Although the projection scheme in \Cref{lem:Galerkin} ensures structure preservation, the algebraic constraints are, in general, not preserved. 
Constraint-preserving model reduction techniques typically rely on a splitting of the differential and the algebraic part of the system and subsequent reduction of the differential part only. 
For a general overview on model reduction for DAEs, we refer to~\cite{BenS17}.
\end{remark}
%
%
\subsection{Discretization in time}
In this section, we finally discuss two particular time discretization schemes which are dissipation-preserving. 
First, we apply the midpoint rule to~\eqref{eq:model:inputoutput} and show that the discrete energy dissipates as well if the Hamiltonian is quadratic. 
For this, we assume an equidistant partition of the time interval~$[0,T]$ with constant step size~$\tau$, leading to discrete time points $t^n\coloneqq n\tau$. Accordingly, we denote by $x^n$ the approximation of some variable $x$ at time $t^n$. 
Moreover, we make use of the typical notation $x^{n+1/2} = \frac12 (x^n + x^{n+1})$ and $u^{n+1/2}=u(t^{n+1/2})$ for the input. 
The midpoint rule applied to~\eqref{eq:model:inputoutput} then yields 
\begin{subequations}
	\label{eq:midpoint}
	\begin{align}
		\begin{bmatrix} 
			\tau \partial_{z_1}H^{n+1/2} \\ 
			z^{n+1}_2 - z^n_2 \\ 
			0 \end{bmatrix}
		&= 
		\big( \Jmat - \Rmat \big) 
		\begin{bmatrix} 
			z^{n+1}_1 - z^n_1 \\ 
			\tau \partial_{z_2}H^{n+1/2} \\ 
			\tau z_3^{n+1/2} 
		\end{bmatrix}
		+ 
		\tau \begin{bmatrix} 
			\Bmat_1 \\ 
			\Bmat_2 \\ 
			\Bmat_3 
		\end{bmatrix} u^{n+1/2}, \\
		\tau y^{n+1/2}
		&= 
		\begin{bmatrix} 
			\Bmat_1^T & \Bmat_2^T & \Bmat_3^T 
		\end{bmatrix} 
		\begin{bmatrix} 
			z^{n+1}_1 - z^n_1 \\ 
			\tau \partial_{z_2}H^{n+1/2} \\ 
			\tau z_3^{n+1/2} 
		\end{bmatrix}.
	\end{align}
\end{subequations}
The following result states the property of being dissipation-preserving for quadratic Hamiltonians. 
\begin{lemma}[discrete energy dissipation, midpoint rule] 
\label{lem:midpoint}	
Let the Hamiltonian be quadratic as in~\eqref{eq:quadHamiltonian} and set $H^n \coloneqq H(z^n)$. 
Then the midpoint scheme~\eqref{eq:midpoint} satisfies $H^{n+1}-H^n\le \tau\, \langle y^{n+1/2}, u^{n+1/2} \rangle$ and, in particular, $H^{n+1} \le H^n$ for vanishing inputs. 
\end{lemma}

\begin{proof}
We first note that 
\[
	2 H(z) 
	= \langle z_1, \Mmat_1 z_1 \rangle + \langle z_2, \Mmat_2 z_2 \rangle
	= \langle z_1, \partial_{z_1} H\rangle + \langle z_2, \partial_{z_2} H\rangle
	= \langle z, \partial_{z} H\rangle.
\]
With this, we conclude that 
\[
	H^{n+1} - H^n
	= \big\langle z^{n+1} - z^n, \partial_{z}H^{n+1/2} \big\rangle
	= \Bigg\langle \begin{bmatrix} z^{n+1}_1 - z^n_1 \\ \partial_{z_2}H^{n+1/2} \\ z_3^{n+1/2} \end{bmatrix}, \begin{bmatrix} \partial_{z_1}H^{n+1/2} \\ z^{n+1}_2 - z^n_2 \\ 0 \end{bmatrix} \Bigg\rangle. 
\]
Together with~\eqref{eq:midpoint}, this yields 
\begin{align*}
	\tau\, \big( H^{n+1} - H^n \big)
	= \Bigg\langle \begin{bmatrix} z^{n+1}_1 - z^n_1 \\ \tau\partial_{z_2}H^{n+1/2} \\ \tau z_3^{n+1/2} \end{bmatrix}, \begin{bmatrix} \tau\partial_{z_1}H^{n+1/2} \\ z^{n+1}_2 - z^n_2 \\ 0 \end{bmatrix} \Bigg\rangle 
	\le \big\langle \tau y^{n+1/2}, \tau u^{n+1/2} \big\rangle,
\end{align*}
which completes the proof. 
\end{proof}
The second discretization scheme is based on discrete gradients and does not require the Hamiltonian to be a quadratic function of the state.
Discrete gradients have been originally introduced in \cite{Gon96} for Hamiltonian systems and are also used for discretizing dissipative and pH systems, cf.~\cite{McLQR99,CelH17,KinTBK23,FroGLM24,Sch24}.
A continuous map $\overline{\nabla} H\colon \R^n\times\R^n\to\R^n$ with $n=n_1+n_2$ is called a discrete gradient of $H$ if it satisfies
\begin{equation}
	\label{eq:discretegradientproperties}
	\overline{\nabla} H(z,z)  = z\quad\text{and}\quad \big\langle \overline{\nabla} H(z,z'), z'- z \big\rangle = H(z')-H(z)\quad \text{for all }z,z'\in\R^n.
\end{equation}
Examples for the explicit construction of discrete gradients are presented, e.g., in \cite{McLQR99,Eid22}.
In the following, we use the notation $\overline{\partial_{z_1}}H$ and $\overline{\partial_{z_2}}H$ to denote the first $n_1$ and the last $n_2$ entries of $\overline{\nabla} H$, respectively, and we propose the time discrete system
\begin{subequations}
\label{eq:discreteGrad}
\begin{align}
	\begin{bmatrix} \tau \overline{\partial_{z_1}}H\Big(\begin{bsmallmatrix}z_1^n\\ z_2^n\end{bsmallmatrix},\begin{bsmallmatrix}z_1^{n+1}\\ z_2^{n+1}\end{bsmallmatrix}\Big) \\ z^{n+1}_2 - z^n_2 \\ 0 \end{bmatrix}
	&= \big( \Jmat - \Rmat \big) 
	\begin{bmatrix} 
		z^{n+1}_1 - z^n_1 \\ 
		\tau \overline{\partial_{z_2}}H\Big(\begin{bsmallmatrix}z_1^n \\ z_2^n \end{bsmallmatrix}, \begin{bsmallmatrix} z_1^{n+1}\\ z_2^{n+1} \end{bsmallmatrix}\Big) \\ 
	\tau z_3^{n+1/2} 
	\end{bmatrix}
	+ 
	\tau 
	\begin{bmatrix} 
		\Bmat_1 \\ 
		\Bmat_2 \\ 
		\Bmat_3 
	\end{bmatrix} 
	u^{n+1/2}, \\
	\tau y^{n+1/2}
	&= 
	\begin{bmatrix} 
		\Bmat_1^T & 
		\Bmat_2^T & 
		\Bmat_3^T 
	\end{bmatrix} 
	\begin{bmatrix} 
		z^{n+1}_1 - z^n_1 \\ 
		\tau \overline{\partial_{z_2}}H\Big(\begin{bsmallmatrix}z_1^n\\ z_2^n\end{bsmallmatrix},\begin{bsmallmatrix}z_1^{n+1}\\ z_2^{n+1}\end{bsmallmatrix}\Big) \\ 
		\tau z_3^{n+1/2} 
	\end{bmatrix}.
	\end{align}
\end{subequations}
Here, we use for simplicity the midpoint approximation $u^{n+1/2}$ as input. Note, however, that the following lemma and its proof can be easily extended to other consistent time discretizations of $u$.
\begin{lemma}[discrete energy dissipation, discrete gradient method]
\label{lem:discreteGrad}	
The discrete gradient scheme~\eqref{eq:discreteGrad} satisfies $H^{n+1}-H^n\le \tau\, \langle y^{n+1/2}, u^{n+1/2} \rangle$ and, in particular, $H^{n+1} \le H^n$ for vanishing inputs. 
\end{lemma}
\begin{proof}
We use the second equation in \eqref{eq:discretegradientproperties} and proceed similarly as in the proof of \Cref{lem:midpoint} to obtain
\begin{align*}
	\tau\, (H^{n+1} - H^n) 
	&= \tau\, \Big\langle \overline{\nabla} H\Big(\begin{bsmallmatrix}z_1^n\\ z_2^n\end{bsmallmatrix},\begin{bsmallmatrix}z_1^{n+1}\\ z_2^{n+1}\end{bsmallmatrix}\Big), \begin{bsmallmatrix}z_1^{n+1}-z_1^n\\ z_2^{n+1}-z_2^n\end{bsmallmatrix} \Big\rangle\\
	&= \Bigg\langle \begin{bmatrix} z^{n+1}_1 - z^n_1 \\ \tau \overline{\partial_{z_2}}H\Big(\begin{bsmallmatrix}z_1^n\\ z_2^n\end{bsmallmatrix},\begin{bsmallmatrix}z_1^{n+1}\\ z_2^{n+1}\end{bsmallmatrix}\Big)\\ \tau z_3^{n+1/2} \end{bmatrix}, \begin{bmatrix} \tau \overline{\partial_{z_1}}H\Big(\begin{bsmallmatrix}z_1^n\\ z_2^n\end{bsmallmatrix},\begin{bsmallmatrix}z_1^{n+1}\\ z_2^{n+1}\end{bsmallmatrix}\Big) \\ z^{n+1}_2 - z^n_2\\ 0 \end{bmatrix} \Bigg\rangle\\
	&\le \big\langle \tau y^{n+1/2}, \tau u^{n+1/2} \big\rangle.
	\qedhere
\end{align*}
\end{proof}
After the introduction of the energy-based framework and the discussion of important properties, we now turn to specific examples. 
%
%
\section{Collection of Examples}\label{sec:examples}
We start with examples which do not need the third variable~$z_3$. This includes classical input--state--output pH systems as well as the equations of poroelasticity. For semi-explicit DAEs of index~$2$, we will see that~$z_3$ allows the direct formulation without the need of an index reduction. Afterwards, we consider circuit models and constrained mechanical systems. Finally, we give three (constrained) PDE examples, which fit in the given framework if the matrices are replaced by operators with analogous properties. 
%
%
\subsection{Example I: pH systems and gradient systems}
Consider an input--state--output pH system (without constraints) as in~\cite[Sect.~4.2]{SchJ14}. 
Without feedthrough term, such a system has the form 
\begin{align*}
	\dot x 
	&= (\Jmat - \Rmat)\, \partial_x H(x) + \Bmat u,\\
	y 
	&= \Bmat^T \partial_x H(x).
\end{align*}
This directly translates into the form~\eqref{eq:model:inputoutput} with $z_1 = \bullet$ (no $z_1$ variable needed), $z_2 = x$, and $z_3 = \bullet$. 
On the other hand, one may consider a gradient system as in~\cite{EggHS21}, i.e., 
\[
	\Cmat(u) \dot u 
	= - \partial_u H(u) 
\]
with $\Cmat$ being symmetric positive definite. In this case, we reach the form~\eqref{eq:model:inputoutput} with the state $z = [u; \bullet; \bullet]$, no in- or outputs, and $\Jmat=0$, $\Rmat=\Cmat$. 
If $\Cmat$ is skew-symmetric instead, then we have a Hamiltonian system and set $\Jmat=\Cmat$, $\Rmat=0$. 
%
%
\subsection{Example II: poroelasticity}
We now turn to a first DAE example for which the introduced setting is more suitable than the classical pH framework. A spatial discretization of the linear poroelasticity equations~\cite{Bio41,Sho00} yields the DAE 
\begin{align}
	\label{eq:poro:semiDiscrete}
	\begin{bmatrix} 0 & 0 \\ \Dmat & \Cmat \end{bmatrix}
	\begin{bmatrix} \dot u \\ \dot p \end{bmatrix}
	= 
	\begin{bmatrix} -\Amat & \phantom{-}\Dmat^T \\ 0 & -\Bmat \end{bmatrix}
	\begin{bmatrix} u \\ p \end{bmatrix}
	+ 
	\begin{bmatrix} \f \\ \g \end{bmatrix}
\end{align}
with initial conditions $u(0) = u^0$, $p(0) = p^0$. 
Here, the unknowns $u$ and $p$ model the deformation of the material and the pressure of the fluid, respectively.  
The stiffness matrix $\Amat\in \R^{n_u, n_u}$ describes the elasticity in the mechanical part of the problem. Assuming a suitable discretization, which also includes the (homogeneous) boundary conditions, $\Amat$ is symmetric positive definite. The same applies for $\Bmat\in \R^{n_p, n_p}$, which equals the stiffness matrix related to the fluid flow. Moreover, $\Cmat\in \R^{n_p, n_p}$ is a mass matrix which describes the compressibility of the fluid under pressure and is also assumed to be symmetric positive definite. Finally, $\Dmat\in \R^{n_p, n_u}$ is a rectangular matrix of full (row) rank. 
The potential energy, which is relevant from a physical point of view, reads
\begin{align}
	\label{eq:poro:energy}
	H(u,p)
	= \tfrac12\, \big\langle u,\Amat u\big\rangle + \tfrac12\, \big\langle p,\Cmat p\big\rangle.
\end{align}
%

In~\cite{AltMU21c}, the pH structure of~\eqref{eq:poro:semiDiscrete} has been discussed. One possibility is to consider an extended system with the additional variable $w=\dot u$. With this, the system can be written in the form
\begin{equation*}
	\begin{bmatrix} 
		0 & 0 & 0\\ 0 & \Amat & 0 \\ 
		0 & 0 & \Cmat 
	\end{bmatrix}
	\begin{bmatrix} 
		\dot{w} \\ 
		\dot{u}\\ 
		\dot{p} 
	\end{bmatrix}
	= 
	\begin{bmatrix} 
		0 & -\Amat & \Dmat^T \\ 
		\Amat & 0 & 0 \\ 
		-\Dmat & 0 & -\Bmat 
	\end{bmatrix}
	\begin{bmatrix} 
		w \\ 
		u \\ 
		p 
	\end{bmatrix}
	+ 
	\begin{bmatrix} 
		f \\ 
		0 \\ 
		g 
	\end{bmatrix},
\end{equation*}
which has the typical pH structure $\Emat \dot z = (\Jmat-\Rmat) z + h$ with $z = [w;u;p]$, $\Emat = \Diag(0,\Amat,\Cmat)$, 
\[
	\Jmat 
	= 
	\begin{bmatrix} 
		0 & -\Amat & \Dmat^T \\ 
		\Amat & 0 & 0 \\ 
		-\Dmat & 0 & 0 
	\end{bmatrix}
	= -\Jmat^T, \qquad
	\Rmat
	= 
	\begin{bmatrix} 
		0 & 0 & 0 \\ 
		0 & 0 & 0 \\ 
		0 & 0 & \Bmat 
	\end{bmatrix} 
	= \Rmat^T
	\ge 0.
\]
The corresponding Hamiltonian $H = \frac12\, \langle z, \Emat z\rangle$ coincides with the energy~\eqref{eq:poro:energy} and satisfies (for vanishing right-hand sides) $\ddt H = \langle z, \Emat \dot z\rangle = -\langle z, \Rmat z\rangle = - \langle p, \Bmat p \rangle\le 0$. 

An alternative formulation, also presented in~\cite{AltMU21c}, can be obtained by the introduction of a more artificial variable~$q$ via $\Bmat q = \Dmat u + \Cmat p$. This then leads to a pH formulation with a different energy. We would like to emphasize that both formulations require an extension of the system equations in order to obtain the desired pH structure. This is not desirable, since the involved matrices are usually large as they result from a spatial discretization of the original PDE system. 

The energy-based framework~\eqref{eq:model:inputoutput}, in turn, allows a direct formulation without any extensions. For this, we set $z = [u; \Cmat p; \bullet]$. Due to $\partial_{z_1} H = \Amat u$ and $\partial_{z_2} H = \Cmat^{-1}z_2 = p$, we can rewrite~\eqref{eq:poro:semiDiscrete} as
\[
	\begin{bmatrix} \Amat u\\ \Cmat\dot{p} \end{bmatrix}
	= \begin{bmatrix} 0 & \Dmat^T \\ - \Dmat & -\Bmat \end{bmatrix}
	\begin{bmatrix} \dot u \\ p \end{bmatrix}
	+ \begin{bmatrix} f \\ g \end{bmatrix},
\]
where the right-hand sides serve as input. The dissipation inequality from \Cref{lem:dissipative} then directly implies~$\ddt H = - \langle p, \Bmat p\rangle$ for the homogeneous case. 
\begin{remark}
If the permeability is allowed to depend on the porosity of the material, the diffusion term in the second equation of~\eqref{eq:poro:semiDiscrete} depends (in a nonlinear way) on the displacement~$u$, cf.~\cite{CaoCM13,AltM22}. 
Hence, the constant matrix $\Bmat$ from the linear model turns into $\Bmat(u)$. 
%
Assuming the same assumptions as in the linear case with $\Bmat(u)$ being symmetric and (uniformly) positive definite, we recover the structure from the linear case.  Hence, also nonlinear poroelasticity can be written in the form~\eqref{eq:model:inputoutput}, cf. \Cref{rem:statedependentJR}.  
\end{remark}
%
%
\subsection{Example III: a class of index-$1$ DAEs}
In the previous section, we have discovered that the equations of poroelasticity provide a DAE, where the pH formulation does not seem to be the natural choice of modeling. 
In fact, DAEs where the Hamiltonian equals a quadratic norm of the entire state (i.e., $H(z) = \frac12 z^T\Mmat z$ with $\Mmat=\Mmat^T>0$) cannot be modeled as a pH system without extending the system equations, since the energy would contain a singular matrix~$\Emat$ extracting parts of the state. A particular class of such DAEs (here without in- or outputs) has the form 
\begin{align*}
	\Dmat \dot z_1 + \dot z_2 
	&= (\Jmat_2-\Rmat_2)\, \Mmat_2 z_2, \\
	0
	&= \Mmat_1 z_1 - \Dmat^T \Mmat_2 z_2
\end{align*}
with $\Jmat_2=-\Jmat_2^T$, $\Rmat_2=\Rmat_2^T\ge0$ and $\Mmat_1=\Mmat_1^T>0$, $\Mmat_2=\Mmat_2^T>0$. 
Note that the first equation gives a differential equation for $z_2$, whereas the second equation provides an algebraic equation which can be solved for $z_1$. 
Hence, the system is of index~$1$; see~\cite[Sect.~3.3]{KunM06} for a definition of the index. We proceed with defining the state $z=[z_1;z_2;\bullet]$ and the quadratic energy 
\[
	H(z_1,z_2) 
	= \tfrac12\, \big\langle z_1, \Mmat_1 z_1\big\rangle + \tfrac12\, \big\langle z_2, \Mmat_2 z_2\big\rangle.
\] 
Due to $\partial_{z_1}H = \Mmat_1 z_1$ and $\partial_{z_2}H = \Mmat_2 z_2$, the above DAE can be written in the form~\eqref{eq:model:inputoutput}, namely 
\[
	\begin{bmatrix} 
		\Mmat_1 z_1 \\ 
		\dot{z}_2 
	\end{bmatrix}
	= 
	\bigg( 
	\begin{bmatrix} 
		0 & \Dmat^T \\ 
		-\Dmat & \Jmat_2 
	\end{bmatrix}
	-
	\begin{bmatrix} 
		0 & 0 \\ 
		0 & \Rmat_2 
	\end{bmatrix}
	\bigg)
	\begin{bmatrix} 
		\dot{z}_1 \\ 
		\Mmat_2 z_2 
	\end{bmatrix}.
\]
%
%
\subsection{Example IV: semi-explicit index-$2$ DAEs}
Consider a semi-explicit DAE  
\begin{subequations}
\label{eq:semiexplicitIndex2DAE}
\begin{align}
	\Mmat \dot u + \Amat u + \Bmat^T \lambda 
	&= f, \\
	\Bmat u \hspace{3.2em} 
	&= g
\end{align}
\end{subequations}
with symmetric positive definite matrices~$\Mmat$, $\Amat$ and a full-rank matrix~$\Bmat$. 
In this case, we have a system of index~$2$; see~\cite[Ch.~VII.1]{HaiW96}.
Considering as Hamiltonian $H = \frac12 \langle u, \Mmat u\rangle$, and setting $z_1=\bullet$, $z_2=\Mmat u$, and~$z_3=\lambda$, we get
\[
	\begin{bmatrix} 
		\Mmat \dot u \\ 
		0 
	\end{bmatrix}
	= 
	\begin{bmatrix} 
		-\Amat & -\Bmat^T \\
		\Bmat & 0 
	\end{bmatrix}
	\begin{bmatrix} 
		u \\ 
		\lambda 
	\end{bmatrix}
	+
	\begin{bmatrix}
		f\\
		-g
	\end{bmatrix}
	.
\]
Again, the right-hand sides serve as input and the system may be completed by an appropriate output equation.

For numerical reasons, one may be interested in related systems of lower index. We present two approaches and start with the introduction of a singular perturbation in the second equation. 
More precisely, we may replace the original DAE by the system 
\begin{align*}
	\Mmat \dot u + \Amat u + \Bmat^T \lambda &= f, \\
	\eps^2 \dot \lambda + \Bmat u \hspace{3.2em} &= g
\end{align*}
with a small parameter $1 \gg\eps>0$. 
Note that this system equals an ODE. 
To bring this into the form~\eqref{eq:model:inputoutput}, we introduce $z = [u; \eps \lambda;\bullet]$ as state and the Hamiltonian~$H = \frac12 \langle u, \Amat u\rangle + \langle u, \Bmat^T\lambda\rangle$. 
This leads to $\partial_{z_1} H = \Amat u + \Bmat^T\lambda$, $\partial_{z_2} H = \frac1\eps \Bmat u$, and the system equations 
\[
	\begin{bmatrix}
		\Amat u + \Bmat^T\lambda \\ 
		\eps \dot\lambda 
	\end{bmatrix}
	= 
	\begin{bmatrix} 
		-\Mmat & 0 \\ 
		0 & -\Imat 
	\end{bmatrix}
	\begin{bmatrix} 
		\dot{u} \\ 
		\eps^{-1} \Bmat u 
	\end{bmatrix}
	+ 
	\begin{bmatrix} 
		f \\ 
		\eps^{-1} g 
	\end{bmatrix}
\]
with $\Imat$ denoting the identity matrix. 
An alternative way to reduce the index is to replace the constraint by its derivative~$\Bmat \dot{u} = \dot{g}$. 
In this case, one may consider the state $z=[u;\bullet;\lambda]$ together with the energy functional $H = \frac12 \langle u, \Amat u\rangle$. 
This then gives 
\[
	\begin{bmatrix} 
		\Amat u \\ 
		0 
	\end{bmatrix}
	= 
	\begin{bmatrix} 
		-\Mmat & -\Bmat^T \\ 
		\Bmat & 0 
	\end{bmatrix}
	\begin{bmatrix} 
		\dot{u} \\ 
		\lambda 
	\end{bmatrix}
	+ 
	\begin{bmatrix} 
		f \\ 
		-\dot g 
	\end{bmatrix}. 
\]
We would like to emphasize that the index-reduced models are based on Hamiltonians which may not reflect the physical energy of the system.
%
%
\subsection{Example V: DC power network}
Another example, for which the extension of the state by a third component is necessary, is the linear electrical circuit from~\cite[Sect.~4.1]{MehM19}. 
For given parameters $R_G, R_L, R_R, L, C_1, C_2 > 0$ and a voltage source $E_G$, the model equations read
\begin{subequations}
\label{eq:DCexample}
\begin{align}
	L\dot{I} 
	&= -R_L I + V_2 - V_1, \\
	C_1\dot{V}_1 
	&= I - I_G, \\
	C_2\dot{V}_2 
	&= - I - I_R, \\
	0 
	&= - R_G I_G + V_1 + E_G, \\
	0 
	&= - R_R I_R + V_2 
\end{align}
\end{subequations}
with unknowns $I$, $V_1$, $V_2$, $I_G$, $I_R$. 
The corresponding energy has the form 
\[
	H(I, V_1, V_2)
	= \tfrac12\, L I^2 + \tfrac12\, C_1 V_1^2 + \tfrac12\, C_2 V_2^2.
\]
In~\cite{MehM19}, it is shown that this system can be written as a pH-DAE. More precisely, summarizing the unknowns in a vector~$x$, there exist $\Jmat=-\Jmat^T$ and $\Rmat = \Rmat^T \ge 0$ such that~\eqref{eq:DCexample} is equivalent to 
\[
	\Emat \dot{x} 
	= (\Jmat - \Rmat)x + \Bmat u
	= (\Jmat - \Rmat)x + e_4 E_G
\]  
with $\Emat = \Diag(L, C_1, C_2, 0, 0)$ and output $y = \Bmat^T x$.
Defining the states~$z_1 = \bullet$, $z_2 = [LI; C_1V_1; C_2V_2]$, and $z_3=[I_G; I_R]$, we have 
\[
	\begin{bmatrix} 
		\dot{z}_2 \\ 
		0 
	\end{bmatrix}
	= 
	\Emat \dot{x}, \qquad
	\begin{bmatrix} 
		\partial_{z_2}H \\ 
		z_3 
	\end{bmatrix}
	= x.
\]
Hence, we directly obtain the structure~\eqref{eq:model:inputoutput} with the same matrices~$\Jmat$ and $\Rmat$. 
%
%
\subsection{Example VI: nonlinear circuit model}
In \cite[eq.~(13)]{GerHRS21}, the authors present a general nonlinear circuit model of the form  
\begin{align*}
	A_{\mathrm{C}}\dot{q}_{\mathrm{C}}+A_{\mathrm{R}}G \big(A_{\mathrm{R}}^T\phi\big)+A_{\mathrm{L}}i_{\mathrm{L}}+A_{\mathrm{S}}i_{\mathrm{S}} 
	&= 0,\\
	-A_{\mathrm{L}}^T\phi+\dot{\psi}_{\mathrm{L}} 
	&=0,\\
	-A_{\mathrm{S}}^T\phi+u_{\mathrm{S}} 
	&=0,\\
	A_{\mathrm{C}}^T\phi-\nabla H_{\mathrm{C}}(q_{\mathrm{C}}) 
	&=0,\\
	i_{\mathrm{L}}-\nabla H_{\mathrm{L}}(\psi) 
	&=0
\end{align*}
with matrices $A_{\mathrm{C}}$, $A_{\mathrm{R}}$, $A_{\mathrm{L}}$, $A_{\mathrm{S}}$ representing edges connected to capacitances, resistances, inductances, and sources, respectively. 
Moreover, $q_{\mathrm{C}}$ contains the charges at the capacitances, $G$ is a function describing the current-voltage relation at the resistances satisfying $\langle \eta, G(\eta)\rangle\ge 0$ for all vectors $\eta$ of appropriate dimension, $\phi$ contains the vertex potentials, $i_{\mathrm{L}}$ and $i_{\mathrm{S}}$ are the currents at the inductances and sources, respectively, $\psi_{\mathrm{L}}$ the magnetic fluxes at the inductances, $u_{\mathrm{S}}$ the voltages at the sources, and $H_{\mathrm{C}}$ and $H_{\mathrm{L}}$ denote the energy stored in the capacitances and inductances, respectively.

In the following, we eliminate the last equation and the variable $i_{\mathrm{L}}$. 
We further consider the special case where $G$ is linear and can be represented by a positive semi-definite matrix. 
As Hamiltonian, we consider the total energy $H(q_{\mathrm{C}},\psi) = H_{\mathrm{C}}(q_{\mathrm{C}})+H_{\mathrm{L}}(\psi)$.
Then, introducing $z_1=q_{\mathrm{C}}$, $z_2=\psi_{\mathrm{L}}$, and $z_3=[i_{\mathrm{S}}; \phi]$, we may write the system in the structured form 
\begin{equation*}
	\begin{bmatrix}
		\nabla H_{\mathrm{C}} \\ 
		\dot{\psi}_{\mathrm{L}} \\ 
		0 \\ 
		0
	\end{bmatrix}
	=
	(\Jmat-\Rmat)
	\begin{bmatrix}
		\dot{z}_1\\
		\partial_{z_2} H\\
		z_3
	\end{bmatrix}
	+ 
	\Bmat u_{\mathrm{S}}
	=
	\begin{bmatrix}
		0 & 0 & 0 & A_{\mathrm{C}}^T\\
		0 & 0 & 0 & A_{\mathrm{L}}^T\\
		0 & 0 & 0 & A_{\mathrm{S}}^T\\
		-A_{\mathrm{C}} & -A_{\mathrm{L}} & -A_{\mathrm{S}} & -A_{\mathrm{R}}GA_{\mathrm{R}}^T
	\end{bmatrix}
	\begin{bmatrix}
		\dot{q}_{\mathrm{C}}\\
		\nabla H_{\mathrm{L}}\\
		i_{\mathrm{S}}\\
		\phi
	\end{bmatrix}
	-
	\begin{bmatrix}
		0 \\
		0 \\
		u_{\mathrm{S}} \\
		0
	\end{bmatrix}.
\end{equation*}
%
%
%
%
\subsection{Example VII: constrained mechanical systems}
Mechanical systems are typically modeled by differential equations of second order, often in combination with constraints. 
If we consider nonholonomic constraints on velocity level, the system equations read 
\begin{align*}
	\Mmat\ddot{x} + \Dmat\dot{x} + \Kmat x + \Bmat^T\lambda 
	&= f, \\
	\Bmat \dot x \hspace{5.95em}
	&= g, 
\end{align*}
which equals a DAE of index~$2$. 
Therein, the mass and stiffness matrices $\Mmat$ and $\Kmat$ are assumed to be symmetric positive definite, whereas the damping matrix $\Dmat$ is symmetric positive semi-definite. 
For the corresponding first-order formulation, we introduce the variable $y\coloneqq\dot{x}$. 
With this, the combination of kinetic and potential energy reads 
\[
	H 
	= \tfrac 12\, \big\langle y,\Mmat y\big\rangle + \tfrac 12\, \big\langle x, \Kmat x\big\rangle.
\]
Introducing $z_1 = \bullet$, $z_2 = [x; \Mmat y]$, and $z_3 = \lambda$ as states, we get $\partial_{z_2} H = [\Kmat x; y]$, leading to 
\begin{align*}
	\begin{bmatrix}
		\dot{x} \\ 
		\Mmat \dot{y} \\ 
		0 
	\end{bmatrix}
	=
	\begin{bmatrix}
		\dot{z}_2 \\ 
		0 
	\end{bmatrix}
	=
	(\Jmat-\Rmat)
	\begin{bmatrix} 
		\partial_{z_2} H \\ 
		z_3
	\end{bmatrix}
	+
	\begin{bmatrix}
		0 \\ 
		f \\ 
		-g 
	\end{bmatrix}
	=
	\begin{bmatrix}
		0 & \Imat & 0 \\ 
		-\Imat & -\Dmat & -\Bmat^T \\ 
		0 & \Bmat & 0 
	\end{bmatrix}
	\begin{bmatrix}
		\Kmat x \\ 
		y \\ 
		\lambda \\ 
	\end{bmatrix}
	+
	\begin{bmatrix}
		0 \\ 
		f \\ 
		-g \\ 
	\end{bmatrix}
\end{align*}
with $\Imat$ denoting again the identity matrix. 

If we consider a mechanical system with holonomic constraints, i.e., with constraints on $x$ rather than $\dot x$, we obtain a DAE of index~$3$; see \cite[Ch.~VII.1]{HaiW96}. 
In this case, one often performs an index reduction to overcome numerical instabilities.
The corresponding GGL-formulation~\cite{GeaGL85} includes the derivative of the constraint and an additional Lagrange multiplier. 
One variant reads 
\begin{align*}
	\Kmat \dot x - \Kmat y \hspace{2.62em} + \Bmat^T\mu  
	&= 0, \\
	\Mmat\dot y + \Dmat y + \Kmat x + \Bmat^T\lambda 
	&= f, \\
	\Bmat x \hspace{3.25em}
	&= g, \\
	\Bmat y \hspace{5.9em}
	&= \dot g.
\end{align*}
To bring this into the form~\eqref{eq:model:inputoutput}, we set $z_1 = \bullet$, $z_2 = [\Kmat x; \Mmat y]$, and $z_3 = [\lambda; \mu]$. 
Considering the same energy as before, this implies $\partial_{z_2} H = [x; y]$ such that the system can be written as 
\begin{align*}
	\begin{bmatrix}
		\Kmat \dot{x} \\ 
		\Mmat \dot{y} \\ 
		0 \\ 
		0 
	\end{bmatrix}
	=
	\begin{bmatrix}
		\dot{z}_2 \\ 
		0 
	\end{bmatrix}
	=
	(\Jmat-\Rmat)
	\begin{bmatrix} 
		\partial_{z_2} H\\
		z_3
	\end{bmatrix}
	=
	\begin{bmatrix}
		0 & \Kmat & 0 & -\Bmat^T \\ 
		-\Kmat & -\Dmat & -\Bmat^T & 0 \\
		0 & \Bmat & 0 & 0 \\
		\Bmat & 0 & 0 & 0 
	\end{bmatrix}
	\begin{bmatrix}
		x \\ 
		y \\ 
		\lambda \\ 
		\mu
	\end{bmatrix}
	+
	\begin{bmatrix}
		0 \\ 
		f \\ 
		-\dot{g} \\ 
		-g
	\end{bmatrix}.
\end{align*}

An alternative approach, which comes along without an index reduction, is to add a vanishing term to the energy. More precisely, for a constant~$g$, we define
\[
	H 
	= \tfrac 12\, \big\langle y,\Mmat y\big\rangle + \tfrac 12\, \big\langle x, \Kmat x\big\rangle + \big\langle \lambda, \Bmat x - g\big\rangle, 
\]
which equals the original energy on the solution manifold. 
For this, the choice $z_1 = [x; y; \lambda]$, $z_2 = z_3 = \bullet$ leads to 
\begin{align*}
	\begin{bmatrix}
		\Kmat x + \Bmat^T \lambda \\ 
		\Mmat y \\ 
		\Bmat x - g\\  
	\end{bmatrix}
	=
	\partial_{z_1} H 
	=
	(\Jmat-\Rmat)\, \dot{z}_1 
	=
	\begin{bmatrix}
		-\Dmat & -\Mmat & 0 \\ 
		\Mmat & 0 & 0 \\
		0 & 0 & 0 
	\end{bmatrix}
	\begin{bmatrix}
		\dot{x} \\ 
		\dot{y} \\ 
		\dot{\lambda}
	\end{bmatrix}
	+
	\begin{bmatrix}
		f \\ 
		0 \\  
		0
	\end{bmatrix}.
\end{align*}
The corresponding output then only contains $\dot x$. 
\subsection{Example VIII: vibrating string}
\label{sec:examples:string}
We now turn to PDE examples, starting with the vibrating string in one space dimension~\cite[Sect.~1.1]{JacZ12}.  
Let $\rho$ denote the constant mass density and $T$ the Young modulus of the material.  
Then we seek the (vertical) displacement~$w$ satisfying
\[
	\rho\, \ddot w(x,t) 
	= \partial_x \big( T\partial_x w(x,t) \big)
\]
for all~$x\in(0,1)$ and~$t\in[0,T]$ together with homogeneous Dirichlet boundary conditions. 
The corresponding energy reads 
\[
	H
	= \tfrac12\, \big\langle \dot w, \rho \dot w\big\rangle + \tfrac12\, \big\langle \partial_x w, T \partial_x w \big\rangle
	\coloneqq \frac12 \int_0^1 \rho\, \dot w^2 \dx + \frac12 \int_0^1 T\, (\partial_x w)^2 \dx. 
\]
Note that, leaving the finite-dimensional setting, the brackets $\langle\,\cdot\,,\cdot\,\rangle$ now denote the inner product in $L^2(0,1)$. 
As for the pH formulation (see, e.g., \cite{MehZ24}), one can set $z_2 = [\rho\dot{w}; \partial_x w]$ and $z_1 = z_3 = \bullet$. 
This then yields 
\[
	\begin{bmatrix}
		\rho \ddot {w} \\ 
		\partial_x \dot w
	\end{bmatrix}
	=
	\dot{z}_2
	= 
	(\Jmat-\Rmat)\, 
	\partial_{z_2} H
	=
	\begin{bmatrix} 
		0 & \partial_x \\
		\partial_x & 0 
	\end{bmatrix}
	\begin{bmatrix}
		\dot{w} \\ T \partial_x w 
	\end{bmatrix}.
\]
Hence, we have $\Rmat = 0$ and $\Jmat$ is skew-symmetric, since $\langle \partial_x u, v\rangle = - \langle u, \partial_x v\rangle$ if at least one of the functions $u, v$ has vanishing boundary data. 
%
%
%
\subsection{Example IX: viscoelastic Stokes problem}
Following~\cite{Ren89}, a model for the fluid flow in a structure such as molten polymers is given by 
\begin{align*}
	\rho \dot{v} - \nabla\cdot T + \nabla p &= f, \\
	\eps \dot{T} - \tfrac{\eta}{2}\, \big( \nabla v + (\nabla v)^T \big) + T &= g, \\
	\nabla \cdot v &= 0.
\end{align*}
Therein, the unknowns are the velocity field~$v$ (assumed here to have homogeneous Dirichlet boundary conditions on the boundary of the computational domain~$\Omega$, i.e., $v = 0$ on $\partial\Omega$), the symmetric Cauchy stress tensor $T$, and the pressure~$p$. 
Moreover, we have the density~$\rho$, the zero-shear-rate viscosity $\eta$, and the relaxation time $\eps$ as given parameters of the system. 
As energy function, we consider 
\[
	H 
	= \int_\Omega \frac{\eta\rho}{2}\, |v(x,t)|^2 + \frac{\eps}{2}\, |T(x,t)|^2 \dx.
\]
We show that also this example fits in the framework of~\eqref{eq:model:inputoutput} if the appearing matrices are replaced by (differential) operators. 
Considering $z_1 = \bullet$, $z_2 = [\rho v; \eps T]$, and $z_3 = p$, we get $\partial_{z_2}H = [\eta v; T]$. 
With this, we can write the (homogeneous) viscoelastic Stokes problem in the form  
\begin{align*}
	\begin{bmatrix}
		\rho \dot{v} \\ 
		\eps \dot{T} \\ 
		0
	\end{bmatrix}
	=
	\begin{bmatrix}
		\dot{z}_2 \\ 
		0
	\end{bmatrix}
	=
	(\Jmat-\Rmat)
	\begin{bmatrix} 
		\partial_{z_2} H\\
		z_3
	\end{bmatrix}
	=
	\begin{bmatrix}
		0 & \nabla\cdot & - \nabla \\ 
		\nabla^\text{s} & -\id & 0 \\ 
		-\nabla \cdot & 0 & 0
	\end{bmatrix}
	\begin{bmatrix}
		\eta v\\
		T\\
		p 
	\end{bmatrix},
\end{align*}
%
where~$\nabla^\text{s}$ equals the symmetric gradient. 
Hence, the dissipative part is given by~$\Rmat=\Diag(0,\id,0)$. 
To see that the remaining part is indeed skew-symmetric, we use the homogeneous boundary conditions of $v$ and the symmetry of $T$, leading to 
\[
	2\, \langle \nabla^\text{s} v, T\rangle
	= \langle \nabla v, T\rangle + \big\langle (\nabla v)^T, T\big\rangle
	= -\langle v, \nabla\cdot T\rangle - \big\langle v, \nabla \cdot(T^T)\big\rangle
	= -2\,\langle v, \nabla\cdot T\rangle.
\]
%
%
\subsection{Example X: Cahn--Hilliard equation}
In this final example, we consider the Cahn--Hilliard equation in the weak form. 
In a bounded space--time domain~$\Omega \times [0,T]$ with $\Omega \subset \R^d$, $d\in\{2, 3\}$, and $T<\infty$ the Cahn--Hilliard equations~\cite{CahH58,EllS86} are given by 
\begin{subequations}
	\label{eq:CahnHilliard_bulk}
	\begin{alignat}{3}
		\dot u - \sigma\Delta w 
		&= 0  &&\qquad\text{in } \Omega \times [0,T],\label{eq:CahnHilliard_bulk_a} \\ 
		- \eps\, \Delta u + \eps^{-1} W^\prime(u) 
		&= w &&\qquad \text{in } \Omega \times [0,T]. \label{eq:CahnHilliard_bulk_b}
	\end{alignat}
\end{subequations}
%
Therein, $\eps$ denotes the so-called interaction length and $W$ the energy potential. 
Modeling a material, which does not interact with the surrounding boundary, we consider homogeneous Neumann boundary conditions for $u$ and $w$.
With the outer normal unit vector~$n$, these conditions read~ $\partial_n u = 0$ and $\partial_n w = 0$ on $\partial\Omega \times [0,T]$. 

For the weak formulation, we introduce $\calV \coloneqq H^1(\Omega)$ and the differential operator
\[
	\calK\colon \calV\to\calV^*, \qquad
	\langle \calK u, v \rangle 
	\coloneqq \int_\Omega \nabla u \cdot \nabla v \dx.
\]
Then, system~\eqref{eq:CahnHilliard_bulk} together with the mentioned boundary conditions is equivalent to 
\begin{alignat*}{3}
	\dot u + \sigma \calK w 
	&= 0, \\ 
	\eps\, \calK u + \eps^{-1} W^\prime(u) 
	&= w. 
\end{alignat*}
The corresponding energy reads 
\[ 
	H_{\text{bulk}}(u)
	= \int_\Omega \frac \eps 2\, \nabla u \cdot \nabla u + \frac 1 \eps\, W(u) \dx
	= \frac \eps 2\, \langle \calK u, u \rangle + \frac 1 \eps\, \int_\Omega W(u) \dx.
\]
Considering the state $z=[u;\bullet;w]$, we get the energy-based formulation without in- or outputs, namely
\begin{align*}
	\begin{bmatrix}
		\eps \calK u + \eps^{-1} W'(u) \\ 0
	\end{bmatrix}
	=
	\begin{bmatrix}
		\partial_{z_1} H \\ 0
	\end{bmatrix}
	=
	(\Jmat-\Rmat)
	\begin{bmatrix}
		\dot{z}_1\\	z_3
	\end{bmatrix}
	=
	\Bigg(
	\begin{bmatrix}
		0 & \id \\ -\id & 0
	\end{bmatrix}
	-
	\begin{bmatrix}
		0 & 0 \\ 0 & \sigma\calK
	\end{bmatrix}
	\Bigg)
	\begin{bmatrix}
		\dot{u}\\ w
	\end{bmatrix}.
\end{align*}
\begin{remark}
One may also consider dynamic boundary conditions which additionally reflect the surface energy; see, e.g., \cite{KenEMRSBD01}. 
With $\Gamma \coloneqq \partial\Omega$, the overall energy is then given as 
\[
	H(u)
	= H_{\text{bulk}}(u) + \int_\Gamma \frac \delta 2\, \nabla_\Gamma u \cdot \nabla_\Gamma u + \frac 1 \delta\, W_\Gamma(u) \dx 
\]
with the interaction length on the boundary~$\delta$ and the boundary energy potential~$W_\Gamma$. 
For the particular model introduced in~\cite{LiuW19}, system~\eqref{eq:CahnHilliard_bulk} is closed by 
\begin{alignat*}{3}
	\partial_n w 
	&= 0  &&\qquad\text{on } \Gamma \times [0,T],\\
	\dot u - \Delta_\Gamma w_\Gamma 
	&= 0  &&\qquad\text{on } \Gamma \times [0,T],\\
	- \delta\, \Delta_\Gamma u + \delta^{-1} W_\Gamma^\prime(u) + \eps \partial_n u
	&= w_\Gamma &&\qquad \text{on } \Gamma \times [0,T].  
\end{alignat*}
Hence, one still assumes that the chemical potential~$w$ does not interact with the solid wall, whereas a new boundary chemical potential~$w_\Gamma$ is introduced. 
A possible weak formulation of the system is presented in~\cite{AltZ23}. 
After an additional index reduction, the resulting system can again be written in the form~\eqref{eq:model:inputoutput}.
\end{remark}
%
%
\section{Conclusions}
A novel energy-based model framework is introduced which covers a large number of examples from different fields of application. 
Systems which can be modeled in such a way are energy dissipative and can be interconnected in a structure preserving way. 
Moreover, a discretization by the midpoint rule or a discrete gradient method preserves the dissipativity property for the discrete energy. 
%
%
\section*{Acknowledgments} 
The authors would like to thank Volker Mehrmann and Riccardo Morandin for helpful discussions on the paper. 
Moreover, RA acknowledges support by the Deutsche Forschungsgemeinschaft (DFG, German Research Foundation) - 467107679.  
%
%
\bibliographystyle{alpha}
\bibliography{references}
%
%
%
%
\end{document}

%% file: def.tex
\definecolor{myBlue}{RGB}{30,144,255} 
\definecolor{myGreen}{RGB}{69,169,0} 
\definecolor{myRed}{RGB}{165,12,42} 
\definecolor{myOrange}{RGB}{225,92,22} 
\definecolor{mycolor0}{rgb}{0.12156862745098,0.466666666666667,0.705882352941177} 
\definecolor{mycolor1}{rgb}{0.00000,0.44700,0.74100}
\definecolor{mycolor2}{rgb}{0.85000,0.32500,0.09800}
\definecolor{mycolor3}{rgb}{0.49400,0.18400,0.55600}
\definecolor{mycolor4}{rgb}{0.92900,0.69400,0.12500}
\definecolor{mycolor5}{rgb}{0.46600,0.67400,0.18800}
\definecolor{mycolor6}{rgb}{0.30100,0.74500,0.93300}
\definecolor{mycolor7}{rgb}{0.63500,0.07800,0.18400}

\newcommand{\delete}[1]{ }

\def\R{\mathbb{R}}

\newcommand\dx{\,\mathrm{d}x}

\newcommand{\ddt}{\ensuremath{\frac{\mathrm{d}}{\mathrm{d}t}} }

\DeclareMathOperator{\id}{id}

\DeclareMathOperator{\Diag}{Diag}

\newcommand{\calK}{\ensuremath{\mathcal{K}} }

\newcommand{\calV}{\ensuremath{\mathcal{V}}}

\newcommand{\eps}{\ensuremath{\varepsilon}}

\newcommand{\f}{\ensuremath{f}}
\newcommand{\g}{\ensuremath{g}}
\newcommand{\Amat}{\mathbf{A}}
\newcommand{\Bmat}{\mathbf{B}}
\newcommand{\Cmat}{\mathbf{C}}
\newcommand{\Dmat}{\mathbf{D}}
\newcommand{\Emat}{\mathbf{E}}

\newcommand{\FmatSkew}{\mathbf{F}_\text{skew}}
\newcommand{\FmatSym}{\mathbf{F}_\text{sym}}
\newcommand{\Imat}{\mathbf{I}}
\newcommand{\Jmat}{\mathbf{J}}
\newcommand{\Kmat}{\mathbf{K}}
\newcommand{\Mmat}{\mathbf{M}}

\newcommand{\Rmat}{\mathbf{R}}

\newcommand{\Vmat}{\mathbf{V}}

\DeclareFontFamily{U}{matha}{\hyphenchar\font45}
\DeclareFontShape{U}{matha}{m}{n}{
	<-6> matha5 <6-7> matha6 <7-8> matha7
	<8-9> matha8 <9-10> matha9
	<10-12> matha10 <12-> matha12
}{}
\DeclareSymbolFont{matha}{U}{matha}{m}{n}

\DeclareFontFamily{U}{mathx}{\hyphenchar\font45}
\DeclareFontShape{U}{mathx}{m}{n}{
	<-6> mathx5 <6-7> mathx6 <7-8> mathx7
	<8-9> mathx8 <9-10> mathx9
	<10-12> mathx10 <12-> mathx12
}{}
\DeclareSymbolFont{mathx}{U}{mathx}{m}{n}

\DeclareMathDelimiter{\vvvert} {0}{matha}{"7E}{mathx}{"17}%